\documentclass[oneside]{amsart}

\usepackage[letterpaper,body={13.6cm,22.5cm}, mag=1000]{geometry}
\usepackage{amssymb}
\usepackage{amsthm}
\usepackage{amscd}

\usepackage{enumerate}

\numberwithin{equation}{section}
\theoremstyle{plain}
\newtheorem{cor}[equation]{Corollary}
\newtheorem{lemma}[equation]{Lemma}

\newtheorem{proposition}[equation]{Proposition}

\newtheorem{theorem}[equation]{Theorem}

\theoremstyle{definition}

\newcommand{\dlabel}[1]{\ifmmode \text{\ttfamily \upshape [#1] } \else
{\ttfamily \upshape [#1] }\fi \label{#1}}

\newcommand{\Z}{\operatorname{Z} }

\newcommand{\gen}[1]{\left < #1 \right >}

\newcommand{\normal}{\vartriangleleft}
\newcommand{\normaleq}{\trianglelefteq}

\begin{document}
\setcounter{page}{1}
\title[Finite $p$-groups of conjugate type $\{ 1, p^3 \}$]{ Finite $p$-groups of conjugate type $\{ 1, p^3 \}$}

\author{Tushar Kanta Naik}
\address{Harish-Chandra Research Institute 
         Chhatnag Road, Jhunsi, 
          Allahabad-211 019 
                India \& Homi Bhabha
National Institute, Training School Complex, Anushakti Nagar, Mumbai 400085, India}
\email{mathematics67@gmail.com}
\author{Manoj K. Yadav}
\address{Harish-Chandra Research Institute 
         Chhatnag Road, Jhunsi, 
          Allahabad-211 019 
                India \& Homi Bhabha
National Institute, Training School Complex, Anushakti Nagar, Mumbai 400085, India}
\email{myadav@hri.res.in}

\keywords{Finite p-group, conjugate type, isoclinism, Camina group}
\subjclass[2010]{20D15}

\begin{abstract}
We classify finite $p$-groups, upto isoclinism, which have only two conjugacy class sizes $1$ and $p^3$. It turns out that the  nilpotency class  of such groups is $2$. 

\end{abstract}
\maketitle
\section{Introduction}

A finite group $G$ is said to be of  conjugate type $\{m_1, m_2,\ldots, m_n \}$, if the set of conjugacy class sizes of $G$ is $\{m_1, m_2,\ldots, m_n \}$.  Finite groups of conjugate type $\{ 1, n \}$ were  first investigated by Ito 
\cite{Ito53} in 1953. He proved that if $G$ is of conjugate type $\{ 1, n \}$, then $n$ is a power of some prime $p$ and $G$ is a direct product of a non-abelian Sylow $p$-subgroup and an abelian $p'$-subgroup; in particular $G$ is nilpotent. 
Hence, to understand such groups, it is sufficient to study finite $p$-groups of conjugate type $\{ 1, p^n \}$ for $n \ge 1$. Half a century later, Ishikawa \cite{Ishikawa2002} proved that finite $p$-groups of conjugate type $\{ 1, p^n \}$ 
can have nilpotency class at most $3$. In a different paper \cite{Ishikawa1999}, he classified $p$-groups of conjugate type $\{ 1, p \}$ and $\{ 1, p^2 \}$ upto isoclinism (definition is recalled in Section 2).

 In this paper, we investigate finite $p$-groups of conjugate type $\{1, p^3\}$ and present a classification upto isoclinism.  Surprisingly we found that such groups can not be of nilpotency class $3$. Before stating our results, we 
exihibit some examples.

A finite group $G$ is said to be a \emph{Camina group} if $x^G = x G'$ for all $x \in G - G'$, where $x^G$ denotes the conjugacy class of $x$ in $G$. Let $H$ be a finite Camina $p$-group of nilpotency class $2$ 
with $|H'| = p^m \ge p^3$. Let $A$ be any subgroup of $H'$ of order $p^{m-3}$. Then it is easy to see that $H/A$ is a Camina group of conjugate type $\{1, p^3\}$. That such groups exist, follows from the following 
examples. For  an integer $m \ge 3$, let 
$$\mathcal{H} = 
\begin{Bmatrix}
\begin{bmatrix}
1 & \alpha_1 & \alpha_2\\
0 & 1 & \alpha_3\\
0 & 0 & 1
\end{bmatrix}
: \alpha_1, \alpha_2, \alpha_3\in\mathbb{F}_{p^m}
\end{Bmatrix},$$ 
where $\mathbb{F}_{p^m}$ is a finite field of $p^m$ elements.
This is a Camina $p$-group of nilpotency class $2$ with $|\mathcal{H}'| = p^{m}$. 

For any positive integer $r \ge 1$ and prime $p > 2$, consider the  following group constructed by Ito \cite{Ito53}.
\begin{eqnarray}\label{6eqn1}
G_r  &=& \big{\langle} a_1, \ldots, a_{r+1} \mid [a_i, a_j] = b_{ij}, [a_k, b_{ij}] = 1,\\
 &  &\;\;a_i^p = a_{r+1}^p = b_{ij}^p = 1, 1 \le i < j \le r+1, 1 \le k \le r+1\big{\rangle}.\nonumber
\end{eqnarray}

It follows from \cite[Example 1]{Ito53} that the group $G_r$ defined in \eqref{6eqn1} is a special $p$-group of order $p^{(r+1)(r+2)/2}$ and exponent $p$,  and $|G_r'|  = p^{r(r+1)/2}$. This group has only two different conjugacy 
class sizes, namely $1$ and $p^r$. Thus $G_3$ is of conjugate type $\{1, p^3\}$. For simplicity of notation,  we assume that $G_3$  is generated by $a$, $b$, $c$ and $d$.

In the following theorem we provide a classification of all finite $p$-groups of conjugate type  $\{1, p^3\}$, $p > 2$, upto isoclinism.

\begin{theorem}\label{thm1}
 Let $G$ be a finite $p$-group of conjugate type $\{1, p^3\}$, $p > 2$. Then the nilpotency class of $G$ is $2$ and  $G$ is isoclinic to one of following groups:

(i)   A finite Camina $p$-group of nilpotency class $2$ with commutator subgroup of order $p^3$;

(ii)  The group $G_3$, defined in \eqref{6eqn1} for $r = 3$ ;

(iii) The quotient group $G_3/M$, where $M$ is a normal subgroup of $G_3$ given by  $M = \gen{[a,\; b][c,\; d]}$;

(iv) The quotient group $G_3/N$, where $N$ is a normal subgroup of $G_3$ given by 
 $N = \gen{ [a,\;b][c,\;d],\; [a,\;c][b,\;d]^t}$ with $t$ any fixed integer non-square modulo $p$.
\end{theorem}

We remark that the number of generators of Camina groups occurring in (i) can not be bounded as shown by the group $\mathcal{H}/A$, where $\mathcal{H}$ is defined above and $A$ is a subgroup of $\mathcal{H}'$ of order $p^{m-3}$. This group is minimally generated by $2m$ elements. 
 
\smallskip
Since the nilpotency class of a finite $2$-group of conjugate type $\{1, 2^n\}$ for all $n \ge 1$  is $2$ (see Corollary  \ref{cor1}), classification problem reduces to finite $2$-groups of class $2$. To include the case $p=2$, we consider a more general class of finite $p$-groups of class $2$ and conjugate type $\{1, p^3\}$.

Let $\hat{G}_n$ denote the family consisting of  $(n+1)$-generator non-abelian special $p$-groups $G$ of order $p^{(n+1)(n+2)/2}$. Then it follows that all groups  of this family are of conjugate type $\{1, p^n\}$. It also turns out that  any two groups in $\hat{G}_n$ are isoclinic (see the remark following Lemma \ref{lem8}).  So, all groups in the family $\hat{G}_3$ are of conjugate type  $\{1, p^3\}$, where $p$ is any prime including $2$.  

Let $\hat{\mathcal{G}}_3$ denote the subfamily of $\hat{G}_3$ consisting of $2$-groups. For simplicity of notation,  we assume that a group $\mathcal{G}$ from $\hat{\mathcal{G}}_3$ is minimally generated by the set $\{a, b, c, d\}$.  A magma check shows that this family  has exactly $989$ non-isomorphic groups \cite{Magma}.  

 We are now well prepared to state our next result which provides a classification of $2$-groups of conjugate type $\{1, 8\}$ upto isoclinism.

\begin{theorem}\label{thm2}
 Let $G$ be a finite $2$-group of conjugate type $\{1, 8\}$ and nilpotency class $2$. Then $G$ is isoclinic to one of following groups:
 \begin{enumerate}[(i)]
  \item A finite Camina $2$-group  with commutator subgroup of  order $8$;
  \item A fixed group $\mathcal{G}$ in the family $\hat{\mathcal{G}}_3$, defined above;
  \item The quotient group $\mathcal{G}/M$, where $M$ is a normal subgroup of $\mathcal{G}$ such that  $M = \gen{[a,\; b][c,\; d]}$;
  \item The quotient group $\mathcal{G}/N$, where $N$ is a normal subgroup of $\mathcal{G}$ such that $N = \gen{ [a,\;b][c,\;d],\; [a,\;c][b,\;d][c,\;d]}$.
 \end{enumerate}
\end{theorem}

\smallskip

Unlike $2$-groups, there do exist $p$-groups of nilpotency class $3$ and conjugate type $\{1, p^2\}$. These groups are all isoclinic to the group $W$ given in \cite{Ishikawa1999} and  presented as 
\begin{eqnarray*}
W &=& \big{\langle} a_1, a_2, b, c_1, c_2 \mid [a_1, \;a_2] = b, [a_i, \;b] = c_i\\
& & a_i^p = b^p = c_i^p = [a_i, \;c_i]= [a_1,\;c_2] = [a_2, \;c_1] =1 \;(i = 1,\;2)\big{\rangle}.
\end{eqnarray*}

A natural question which arises here is 
\smallskip

\noindent{\bf Question.} Does there exist a finite $p$-group of nilpotency class $3$, for an odd prime $p$, and conjugate type $\{1, p^n\}$, $n \ge 1$, which is not isoclinic to $W$? If yes, for which values of $n \ge 4$, there exists a  finite $p$-group of nilpotency class $3$ and conjugate type $\{1, p^n\}$?

\smallskip

 Before concluding this section, we set some notations for a multiplicatively written finite group $G$ which are mostly standard. We denote by $G'$ the  commutator subgroup of $G$. For a subgroup $H$ of $G$, by $H^{\#}$ we denote the set 
 of non-trivial elements of $H$. For the elements $x, y, z \in G$, the commutator  $[x, y]$ of $x$ and $y$  is defined by  $x^{-1}y^{-1}xy$, and $[x,y,z]$ = $[[x,y],z]$. Frattini subgroup of $G$ is denoted by $\Phi(G)$. For an element $x \in G$, $x^G$ denotes the conjugacy class of $x$ in $G$. To say that some $H$ is a subgroup or a normal subgroup of $G$ we write  $H \leq G$ or $H \normaleq G$, respectively. 
To indicate, in addition, that $H$ is properly contained in $G$, we write  $H < G$ or $H \normal G$, respectively.

\section{Preliminaries}

The following concept of isoclinism of groups was introduced by P. Hall \cite{Hall40}.

Let $X$ be a finite group and $\bar{X} = X/\Z(X)$. 
Then commutation in $X$ gives a well defined map
$a_{X} : \bar{X} \times \bar{X} \mapsto \gamma_{2}(X)$ such that
$a_{X}(x\Z(X), y\Z(X)) = [x,y]$ for $(x,y) \in X \times X$.
Two finite groups $G$ and $H$ are called \emph{isoclinic} if 
there exists an  isomorphism $\phi$ of the factor group
$\bar G = G/\Z(G)$ onto $\bar{H} = H/\Z(H)$, and an isomorphism $\theta$ of
the subgroup $G'$ onto  $H'$
such that the following diagram is commutative
\[
 \begin{CD}
   \bar G \times \bar G  @>a_G>> G'\\
   @V{\phi\times\phi}VV        @VV{\theta}V\\
   \bar H \times \bar H @>a_H>> H'.
  \end{CD}
\]
The resulting pair $(\phi, \theta)$ is called an \emph{isoclinism} of $G$ 
onto $H$. Notice that isoclinism is an equivalence relation among finite groups.

The following two results follow from \cite{Hall40}.

\begin{proposition}\label{prop4}
Let $G$ and $H$ be two isoclinic finite $p$-groups. Then $G$ and $H$ are of the same conjugate type.
\end{proposition}

\begin{proposition}\label{prop5}
Let $G$ be a finite $p$-group. Then there exists a group $H$ in the isoclinism family of $G$ such that $Z(H) \leq H'$. 
\end{proposition}

Group $H$ which occurred in Proposition \ref{prop5} is called  a \emph{stem group} in its isoclinism class. In the light of the preceding two results, for the  classification of finite $p$-groups of conjugate type $\{1,p^n\}$ upto isoclinism, we only need to consider  a stem group from the respective isoclinism family. 

The following result is due to Vaughan-Lee \cite[p. 270, Theorem]{Vaughan1974}.
\begin{proposition}\label{prop6}
Let  $G$ be a finite $p$-group. Suppose that every conjugacy class of $G$ contains at most $p^n$ elements. Then $|G'|\leq p^{n(n+1)/2}$.
\end{proposition}

The following result is due to Ito \cite[Proposition 3.2]{Ito53}.
\begin{proposition}\label{prop7}
Let $G$ be a finite $p$-group of conjugate type $\{1,p^n\}$. Then the number of elements in any minimal generating set is at least $n$, and order of  subgroup of all the elements of order $p$ of $Z(G)$ is at least $p^n$.
\end{proposition}

Let $G$ be a  finite $p$-group and $x \in G$ be such that $|x^G| = p^{b(x)}$. Then $b(x)$ is called the {\it breadth of} $x$.  The breadth of $G$, denoted by $b(G)$, is defined as $max\{b(x) \mid x \in G\}$.

The following result is due to Parmeggiani and Stellmacher \cite[p. 59, Corollary]{PS99}
\begin{proposition}\label{prop8}
 Let $G$ be a $p$-group, $p > 2$. Then  $b(G) =3$  if and only if one of following holds:
\begin{enumerate}[(i)]
 \item $|G'|$ = $p^3$ and $[G : Z(G)] \geq p^4 $.
 \item $|G'|$ = $p^4$ and there exists $H$ $\vartriangleleft$ $G$ with $|H|$ = $p$ and $[G/H : Z(G/H)] = p^3$.
 \item $|G'| \geq p^4$ and $[G : Z(G)]= p^4 $.
\end{enumerate}
\end{proposition}

A similar result for $p = 2$ is proved by  Wilkens \cite{Wilkens07}, a consequence of which is  stated in the last section for the groups having conjugate type $\{1, 8\}$. 

The following result is a part of \cite[p. 501, Theorem]{Isaacs1970}.
\begin{proposition}\label{prop9}
Let $G$ be a finite group which contains a proper normal subgroup $N$ such that all of the conjugacy classes of $G$ which lie outside of $N$ have  the same lengths. Then either $G/N$ is cyclic or every non-identity element of $G/N$ has prime order.
\end{proposition}

Direct consequence of this result are the following corollaries.
\begin{cor}\label{lem2}
  Let $G$ be a finite $p$-group of conjugate type $\{1, p^n\}$ such  that $Z(G) = G' $. Then 
$ G/Z(G)$ and $G'$ are elementary abelian p-groups.
\end{cor}

\begin{cor}\label{cor1}
Let $G$ be a finite $2$-group of conjugate type $\{1, 2^n\}$, $n \ge 1$. Then the nilpotency class of $G$ is $2$.
\end{cor}

For a given group $G$, an isoclinism $(\phi, \theta)$ from $G$ onto itself is called an {\it autoclinism} of $G$. It is not difficult to prove the following result.

\begin{lemma}\label{ext-autoclinism}
Let $G$ be a group from the family $\hat{G}_n$. Then a bijection between any two minimal generating sets for $G$ extends to an   autoclinism of $G$.
\end{lemma}

For the groups $G := G_{r}$ defined in \eqref{6eqn1}, the following more general result holds true.

\begin{lemma}\label{ext-automorphism}
A bijection between any two minimal generating sets for  $G := G_{r}$,  extends to an automorphism of $G$.
\end{lemma}

\section{Key Lemmas}

We start with the following elementary fact, proof of which is immediate from the Hall-Witt identity.

\begin{lemma}\label{lem1}
  Let $G$ be a group of class $3$ and $x,y,z \in $ $G$ such that  $[x,z], [y,z] \in $ $Z(G)$.  Then $[x,y,z] = 1$.
\end{lemma}

\begin{lemma}\label{lem3}
 Let $G$ be a finite $p$-group of conjugate type $\{1,p^3\}$, $p > 2 $. Then one of following holds: ($i$) $|G'|$ = $p^3$ and $[G : Z(G)]\geq p^4$;
($ii$) $|G'| \geq p^4$ and $[G : Z(G)] = p^4 $.
\end{lemma}
\begin{proof}
 Suppose that there exists a normal subgroup $H$ of $G$ such that $|H|$ = $p$ and $[G/H : Z(G/H)] = p^3$. Then, since  $G$ is of conjugate type $\{1,p^3\}$, it follows that $Z(G/H)$ = $Z(G)/H$. Thus $p^3 = [G/H : Z(G/H)] = [G/H : Z(G)/H] = [G : Z(G)]$. But, for all $x$ $\in$ $G - Z(G)$ we get $[G : C_G(x)]  = p^3$ (as $G$ is of conjugate type $\{1,p^3\}$). Since $x \in C_G(x) - Z(G)$, we have $Z(G) < C_G(x)$; contradicting the equality $[G:Z(G)] = [G : C_G(x)] = p^3$. 
Hence there can not exist any $H$ $\vartriangleleft$ $G$ with $|H|$ = $p$ and $[G/H : Z(G/H)] = p^3$. The proof is now complete from Proposition \ref{prop8}. \hfill $\Box$

\end{proof}

We have noticed above that any group from the family $\hat{G}_{n-1}$ is of conjugate type $\{1,p^{n-1}\}$. The following two results  characterize all finite  $n$-generator special $p$-groups of order $p^{n(n+1)/2 -1}$ and conjugate type $\{1,p^{n-1}\}$.

\begin{lemma}\label{lem4}
Let $G \in \hat{G}_{n-1}$ be a group  generated by $n \geq 4 $ elements $a_1$, $a_2$, $\ldots,$ $a_n$. Suppose that $M < Z(G) = G'$ with $|M|$ = $p$. Then  $G/M$ is of conjugate type $\{1,p^{n-1}\}$ if and only if  $M$ can be reduced to the  form
\begin{equation*}
M = \langle [a_1,\; a_2][a_3,\; a_4][a_5,\; a_6] \dots [a_{2m-1},\; a_{2m}] \rangle, \; \mbox{where} \;  2 \leq m \leq \lfloor n/2 \rfloor.
\end{equation*}
\end{lemma}
\begin{proof}
Notice  that $|G| = p^{n(n+1)/2}$. Also notice that any bijection between two minimal generating sets for $G$ extends to an autoclinism of $G$ by Lemma \ref{ext-autoclinism}. Set $\overline{G}$ = $G/M$; then $|\overline{G}|$ = $p^{(n(n+1)/2)-1}$. Since $G$ is of conjugate type $\{1,p^{n-1}\}$ and $|M| = p$, we have that 
\[Z(G)/M = Z(\overline{G}) = \overline{G}'= \Phi(\overline{G})\] is an elementary abelian $p$-group of order $p^{(n(n-1)/2)-1}$. Thus $[\overline{G} : Z(\overline{G})] = p^n$. Hence $\overline{G}$ is of conjugate 
type $\{1,p^{n-1}\}$ if and only if each  non-central element of $\overline{G}$ commutes only with its own powers up to the central elements.

Let $\bar{x},\; \bar{y} \in  \overline{G} - \Z(\overline{G})$ be such that no one is a power of the other (reading modulo $\Z(\overline{G})$).  Then it is not difficult to see that  $[x,y] \neq 1$ in $G$. Hence, if $[\bar{x} , \bar{y}] =1$ in $\overline{G}$, 
then  $[x,y] \in M^{\#}$.  Any given  central subgroup $M_1$ of order $p$, without loss of generality, can be   written as 
\begin{equation*}
M_1 =\langle [a_1,\; a_2][a_1,\; a_3]^{\alpha_{1,3}} \dots [a_{i},\; a_{j}]^{\alpha_{i,j}} \dots [a_{n-1},\; a_n]^{\alpha_{{n-1},n}} \rangle, 
\end{equation*}
where $1\leq i < j \leq n$.
Now applying the autoclinism induced by the map $a_2 \mapsto a_2a_3^{-\alpha_{1,3}} \dots a_n^{-\alpha_{1,n}}$, $a_i \mapsto a_i$ for $i \ne 2$, $M_1$ gets mapped to 
\begin{equation*}
M_2 =\langle [a_1,\; a_2][a_2,\; a_3]^{\alpha_{2,3}} \dots [a_{i},\; a_{j}]^{\alpha_{i,j}} \dots [a_{n-1},\; a_n]^{\alpha_{{n-1},n}} \rangle 
\end{equation*}
with $2\leq i < j \leq n$ and modified $\alpha_{i,j}$. Notice that $G/M_1$ is isoclinic to  $G/M_2$, and therefore  both $G/M_1$ and $G/M_2$ are of the same conjugate type. We now apply another autoclinism induced by the map  $a_1 \mapsto a_1a_3^{\alpha_{2,3}} \dots a_n^{\alpha_{2,n}}$, $a_i \mapsto a_i$ for $i \ne 1$, and see that $M_2$ gets mapped to  
$$M_3 =\langle [a_1,\; a_2][a_3,\; a_4]^{\alpha_{3,4}} \dots [a_{i},\; a_{j}]^{\alpha_{i,j}} \dots [a_{n-1},\; a_n]^{\alpha_{{n-1},n}} \rangle$$
with $3\leq i < j \leq n$ and modified $\alpha_{i, j}$.

 Take $x = a_1^{j_1}a_2^{j_2} \dots a_n^{j_n}$ and $y = a_1^{k_1}a_2^{k_2} \dots a_n^{k_n}$ be such that none is power of the other (reading modulo $\Z(G)$).  Then $[x,y] \in M_3^{\#}$  only when the following two  conditions  hold true:

 ($i$) $j_3$ = $j_4$ = $\cdots$ = $j_n$ = $k_3$ = $k_4$ = $\cdots$ = $k_n$ = $0$;

 ($ii$) $\alpha_{i,j}$ = $0$,  $\;\;3\leq i < j \leq n$.\\
 Hence $G/M_3$ is of conjugate type $\{1,p^{n-1}\}$ if and only if at least one $\alpha_{i, j}$ is non-zero modulo $p$. We now conclude that  $G/M_3$ is of conjugate type $\{1,p^{n-1}\}$ if and only if
$$M_3 =\langle [a_1,\; a_2][a_3,\; a_4]^{\alpha_{3,4}} \cdots [a_{i},\; a_{j}]^{\alpha_{i,j}} \cdots [a_{n-1},\; a_n]^{\alpha_{{n-1},n}} \rangle$$
with $3\leq i < j \leq n$ and at least one $\alpha_{i, j}$ is non-zero modulo $p$. We can assume that $\alpha_{3,4} \ne 0$.

Make further reductions by applying an autoclinism of $G$ which is a composition of the autoclinisms induced by the maps (1)  $a_4\mapsto a_4^{\alpha_{3,4}^{-1}}$, $a_i \mapsto a_i$ for $i \ne 4$, (2) $a_4 \mapsto a_4a_5^{-\alpha_{3,5}}a_6^{-\alpha_{3,6}} \cdots a_n^{-\alpha_{3,n}}$, $a_i \mapsto a_i$ for $i \ne 4$ and (3) $a_3 \mapsto a_3a_5^{\alpha_{4,5}}a_6^{\alpha_{4,6}} \cdots a_n^{\alpha_{4,n}}$, $a_i \mapsto a_i$ for $i \ne 3$. This action reduces $M_3$ to $M_4$, which is of the following form
$$M_4 =\langle [a_1,\; a_2][a_3,\; a_4][a_5,\; a_6]^{\alpha_{5,6}} \dots [a_{i},\; a_{j}]^{\alpha_{i, j}} \dots [a_{n-1},\; a_n]^{\alpha_{{n-1},n}} \rangle$$
with $5\leq i <  j \leq n$ and modified $\alpha_{i, j}$ again.

Finally, if all $\alpha_{i, j}$ are 0 modulo $p$, we are done by taking $M=M_4$. If not, then finite  repetitions of  the above process reduce $M_4$ to the desired form $M$, completing the proof. \hfill $\Box$
 
\end{proof}

\begin{cor}\label{corollary1}
Let $K$ be an   $n$-generator special $p$-group of order $p^{n(n+1)/2 -1}$ and conjugate type $\{1,p^{n-1}\}$. Then $K$ is isoclinic to $G/M$, where $G = \gen{a_1, a_2, \ldots, a_n}$ $\in$ $\hat{G}_{n-1}$ and $M < Z(G) = G'$ with $|M| = p$ is of the form
$$M = \langle [a_1,\; a_2][a_3,\; a_4][a_5,\; a_6] \cdots [a_{2m-1},\; a_{2m}] \rangle, \; \mbox{where} \;  2 \leq m \leq \lfloor n/2 \rfloor.$$
\end{cor}
\begin{proof}
Notice that the group $K$, given  in the statement, is isomorphic to a quotient  of some group $G$ from the family   $\hat{G}_{n-1}$ by a subgroup of order $p$ contained in $G'$. Now the proof follows from the preceding lemma. \hfill $\Box$

\end{proof}

In particular, if, for an odd prime $p$,  we take a $p$-group $G$ from the class $\hat{G}_{n-1}$ such that the exponent of $G$ is $p$, then $G$ is isomorphic to the group $G_{n-1}$ defined in  \eqref{6eqn1}. Then by Lemma \ref{ext-automorphism}
a bijection between any two minimal generating sets for  $G := G_{n-1}$,  extends to an automorphism of $G$. Therefore, on the lines of the proofs of Lemma \ref{lem4} and Corollary \ref{corollary1} (replacing autoclinism by automorphism and isoclinic by isomorphic), we can prove the following.

\begin{lemma}
Let $K$ be an   $n$-generator special $p$-group of order $p^{n(n+1)/2 -1}$, exponent $p$ and conjugate type $\{1,p^{n-1}\}$, where $p$ is an odd prime. Then $K$ is isomorphic to $G_{n-1}/M$, where $G_{n-1}$ is the group as defined  in  \eqref{6eqn1}  generated by $a_1$, $a_2$, $\ldots,$ $a_n$ and $M < Z(G) = G'$ with $|M| = p$ is of the form
$$M = \langle [a_1,\; a_2][a_3,\; a_4][a_5,\; a_6] \dots [a_{2m-1},\; a_{2m}] \rangle, \; \mbox{where} \;  2 \leq m \leq \lfloor n/2 \rfloor.$$
\end{lemma}

For the case $n =3$, the preceding lemma was proved by Brahana \cite{Brahana40}. For the application point of view, we state it explicitly as a corollary.

\begin{cor}\label{lem6}
Let $K$ be a  $4$-generator special $p$-group of order $p^{9}$, exponent $p$ and conjugate type $\{1,p^{3}\}$, where $p$ is an odd prime.  Then $K$ is isomorphic to $G_3/M$, where  $G_3$ is generated by  $a,b,c$, $d$ and  $M < Z(G) = G'$ with $|M| = p$ is of the form   $H = \gen{[a,\; b][c,\; d]}$.
\end{cor}

\smallskip

Now onward we concentrate only on the groups $G$ from the family $\hat{G}_{3}$.

\begin{lemma}\label{lem7}
Let $G$ be a group from the family $\hat{G}_3$which is generated by  $a,b,c$ and $d$. Suppose that $N < Z(G) = G'$ with $|N| = p^2$. Then $G/N$ is of conjugate type $\{1,p^3\}$ if and only if  $N$ can be reduced to the following form
\begin{equation*}
N = \gen{ [a,\;b][c,\;d],\; [a,\;c][b,\;d]^r},\; \mbox{where}\; r\; \mbox{is any fixed non-square integer modulo}\; p.
\end{equation*}
\end{lemma}
\begin{proof}
Notice  that $|G| = p^{10}$. Set $\overline{G}$ = $G/N$; then $|\overline{G}|$ = $p^8$.  Since $G$ is of conjugate type $\{1,p^3\}$ and $|N| = p^2$, we have that \[Z(G)/N = Z(\overline{G}) =  \overline{G}'= \Phi(\overline{G})\] is an elementary abelian $p$-group of order $p^4$. Thus $[\overline{G} : Z(\overline{G})] = p^4$. Hence $\overline{G}$ is of conjugate type $\{1,p^3\}$ if and only if each non-central element  of $\overline{G}$ commutes only with its own powers up to the central elements.

Let $\bar{x},\; \bar{y} \in  \overline{G} - \Z(\overline{G})$ be such that no one is a power of the other (reading modulo $\Z(\overline{G})$). Then it is not difficult to see that  $[x,y] \neq 1$ in $G$. Hence, 
if $[\bar{x} , \bar{y}] =1$ in $\overline{G}$, then  $[x,y] \in N^{\#}$.  Any given  central subgroup $N_1$ of order $p^2$, without loss of generality,  can be written as one of the following two types:
\begin{eqnarray*}
(i) \;\;\; N_1 &=& \langle [a,\;b][a,\;d]^{i_1}[b,\;c]^{i_2}[b,\;d]^{i_3}[c,\;d]^{i_4},\; [c,\; d] \rangle.\\
(ii) \;\;\; N_1 &=& \langle [a,\;b][a,\;d]^{i_1}[b,\;c]^{i_2}[b,\;d]^{i_3}[c,\;d]^{i_4},\; [a,\;c][a,\;d]^{j_1}[b,\;c]^{j_2}[b,\;d]^{j_3}[c,\;d]^{j_4} \rangle.
\end{eqnarray*}

If $N_1$ is  of type (i), then $\bar{c}$ commutes with $\bar{d}$, although $\bar{c}$ $\notin \langle Z(\overline{G}), \bar{d} \rangle$. Hence $\overline{G}$ can not  be of conjugate type $\{1,p^3\}$. 
Therefore we only need to consider $N_1$ as in type (ii). Now applying the autoclinism induced by the map $a\mapsto a$, $b\mapsto bd^{-i_1}$, $c\mapsto cd^{-j_1}$, $d\mapsto d$, $N_1$ gets mapped to $N_2$, where
\begin{equation*}
N_2 = \langle [a,\;b][b,\;c]^{i_1}[b,\;d]^{i_2}[c,\;d]^{i_3},\; [a,\;c][b,\;c]^{j_1}[b,\;d]^{j_2}[c,\;d]^{j_3} \rangle
\end{equation*}
with modified powers of the basic commutators. Notice that $G/N_1$ and $G/N_2$ are isoclinic.
We now apply another autoclinism induced by the map $a\mapsto ac^{i_1}d^{i_2}$, $b\mapsto b$, $c\mapsto c$, $d\mapsto d$, and see that $N_2$ gets mapped to 
\begin{equation*}
N_3 = \langle [a,\;b][c,\;d]^i,\; [a,\;c][b,\;c]^{j_1}[b,\;d]^{j_2}[c,\;d]^{j_3} \rangle
\end{equation*}
 again with modified powers of commutators.
Note that $i$ is non-zero modulo $p$, otherwise $G/N_3$ can not be of conjugate type $\{1,p^3\}$. Thus the map $a\mapsto a$, $b\mapsto b$, $c\mapsto c$, $d\mapsto d^{i^{-1}}$ extends to an autoclinism of $G$ and maps $N_3$ to 
\begin{equation*}
N_4 = \langle [a,\;b][c,\;d],\; [a,\;c][b,\;c]^{j_1}[b,\;d]^{j_2}[c,\;d]^{j_3} \rangle
\end{equation*}
with modified $j_2$ and $j_3$.
Again note that $j_2$ can not be zero modulo $p$, otherwise $[ab^{j_1}d^{-j_3},\;c] = 1$ and so $G/N_3$ can not be of conjugate type $\{1,p^3\}$. Therefore the map $a\mapsto a$, $b\mapsto b$, $c\mapsto c$,  $d\mapsto c^{-j_1j_2^{-1}}d$ is well  defined. The autoclinism of $G$ induced by this map takes $N_4$ to $N_5$, where, after modifying powers, 
\begin{equation*}
N_5 = \langle [a,\;b][c,\;d],\; [a,\;c][b,\;d]^{i_1}[c,\;d]^{i_2} \rangle.
\end{equation*}

Now let $x = a^{j_1}b^{j_2}c^{j_3}d^{j_4}$ and $y = a^{k_1}b^{k_2}c^{k_3}d^{k_4}$ be such that none is power of the other (reading modulo $\Z(G)$). If $[x,\;y] \in N_5^{\#}$, then at least one of $j_1$ and $k_1$ has to be non-zero modulo $p$.
Without loss of generality we take $j_1$ to be non-zero. Now we can write $y$ as $y = a^{k_1}b^{k_2}c^{k_3}d^{k_4} = (a^{j_1}b^{j_2}c^{j_3}d^{j_4})^{k_1j_1^{-1}}b^{l_2}c^{l_3}d^{l_4}z_1$, where $z_1 \in Z(G)$ and $l_2,\;l_3,\;l_4$ are some suitable integers. So we can modify $x$ and $y$ 
as $x = a^{j_1}b^{j_2}c^{j_3}d^{j_4}$ with $j_1$ non-zero and $y = b^{l_2}c^{l_3}d^{l_4}$. Now $l_2$ has to be non-zero modulo $p$ and $l_4$ has to be $0$. Using similar argument, we can remove power of $b$ in $x$. So we can modify $x$ and
$y$ by $x = a^{j_1}c^{j_3}d^{j_4}$ and $y = b^{l_2}c^{l_3}$. Now $j_3$ has to be $0$. So, finally we have reduced $x$ and $y$ to $x = a^{j_1}d^{j_4}$ and $y = b^{l_2}c^{l_3}$. If $[x,\;y] \in N_5^{\#}$, then $[x^{j_1^{-1}},\;y^{l_2^{-1}}]$ 
also belongs to $N_5^{\#}$. Also $x^{j_1^{-1}} = ad^jz_2$ and $y^{l_2^{-1}} = bc^kz_3$, where $z_2$ and $z_3$ are some central elements. Therefore $[x^{j_1^{-1}},\;y^{l_2^{-1}}]$ = $[a,\;b][a,\;c]^k[b,\;d]^j[c,\;d]^{jk}$. So if 
$[x,\;y] \in N_5^{\#}$, then $[a,\;b][a,\;c]^k[b,\;d]^j[c,\;d]^{jk}$ $\in N_5^{\#}$, and therefore can be written as a product of powers of generators of $N_5^{\#}$. Now comparing power of the basic commutators, we get
\begin{equation*}
j \equiv ki_1 \;(\mbox{mod}\; p)\; \mbox{and}\; jk\equiv ki_2 + 1\; (\mbox{mod}\; p).
\end{equation*}
Solving these we have
\begin{equation*}
k^2i_1 - ki_2 - 1 \equiv 0 \;( \mbox{mod}\; p).
\end{equation*}
This is possible only when $i_2^2$ $+$ $4i_1$ is a square modulo $p$.  From this, we conclude that  $G/N_5$ is of conjugate type $\{1,p^3\}$ if and only if  $N_5$ is of the  following form
\begin{equation*}
N_5 = \langle [a,\;b][c,\;d],\; [a,\;c][b,\;d]^{i_1}[c,\;d]^{i_2} \rangle; \; \mbox{where}\; i_2^2 + 4i_1\; \mbox{is a non-square modulo}\; p.
\end{equation*}

Now we consider two cases, namely: Case 1. $i_2 \ne 0$; Case 2. $i_2  = 0$, and take these one by one.\\

\noindent\textbf{Case 1:}   Let $r$ be a fixed integer non-square modulo $p$. Then $r$ must be non-zero.  Being non-square,  $i_2^2 + 4i_1$ is also non-zero. Thus $\frac {i_2^2 + 4i_1} {4r}$ is a non-zero square modulo $p$.  Thus there exists a 
non-zero $l$ such that  $l^2 = \frac {i_2^2 + 4i_1} {4r}$. Set  $t = \frac {i_2} {2}$.

Now applying the autoclinism of $G$ induced by the map
$a\mapsto a^ld^t$, $b\mapsto b$, $c\mapsto b^tc^l$, $d \mapsto d$, $N_5$ gets mapped to 
\begin{eqnarray*}
N_6  &=& \langle [a,\;b]^l[c,\;d]^l,\; [a,\;b]^{lt}[a,\;c]^{l^2}[b,\;d]^{ti_2 + i_1 - t^2}[c,\;d]^{li_2 - lt} \rangle\\
&=& \langle ([a,\;b][c,\;d])^l,\; ([a,\;b][c,\;d])^{lt}([a,\;c]^{l^2}[b,\;d]^{ti_2 + i_1 - t^2}) \rangle\\
&=& \langle [a,\;b][c,\;d],\; [a,\;c]^{l^2}[b,\;d]^{ti_2 + i_1 - t^2} \rangle\\
&=& \langle [a,\;b][c,\;d],\; [a,\;c]^{\frac {i_2^2 + 4i_1} {4r}}[b,\;d]^{\frac {i_2^2 + 4i_1} {4}} \rangle\\
&=& \langle [a,\;b][c,\;d],\; ([a,\;c][b,\;d]^r)^{\frac {i_2^2 + 4i_1} {4r}} \rangle\\
&=& \langle [a,\;b][c,\;d],\; [a,\;c][b,\;d]^r \rangle\\
&=& N.
\end{eqnarray*}

 Hence we are done in this case.\\

\noindent{\bf Case 2:}   In this case $i_1$ must be a non-square. If $i_1 = r$, then we are done. If not, then  $i_1^{-1}r$ must be a non-zero square, and therefore there exists a non-zero integer $l$ such that   $i_1^{-1}r = l^2$.

 Now the autoclinism of $G$ induced by the map  $a\mapsto a$, $b\mapsto b$, $c\mapsto c^{l^{-1}}$, $d\mapsto d^l$ maps $N_5$ to
\begin{eqnarray*}
N_7 &=&  \langle [a,\;b][c,\;d],\; [a,\;c]^{l^{-1}}[b,\;d]^{li_1} \rangle\\
&=& \langle [a,\;b][c,\;d],\; ([a,\;c][b,\;d]^{l^2i_1})^{l^{-1}} \rangle\\
&=& \langle [a,\;b][c,\;d],\; [a,\;c][b,\;d]^{l^2i_1} \rangle\\
&=& \langle [a,\;b][c,\;d],\; [a,\;c][b,\;d]^r \rangle\\
&=& N.
\end{eqnarray*}
The proof is  now complete.  \hfill $\Box$

\end{proof}

The following result characterizes all $4$-generator special groups of order $p^8$ and conjugate type $\{1,p^{3}\}$.

\begin{cor}\label{corollary2}
Let $K$ be a  $4$-generator  special $p$-group of order  $p^{8}$ and conjugate type $\{1,p^{3}\}$. Then $K$ is isoclinic to $G/N$, where $G = \gen{a,\;b,\;c,\;d}$ $\in$ $\hat{G}_{3}$ and   $N < Z(G) = G'$ with $|N| = p^2$ is of the form
$$N = \gen{ [a,\;b][c,\;d],\; [a,\;c][b,\;d]^r},\; \mbox{where}\; r\; \mbox{is any fixed non-square integer modulo}\; p.$$
\end{cor}
\begin{proof}
Notice that the group $K$, given  in the statement, is isomorphic to a quotient  of some group $G$ from the family   $\hat{G}_3$ by a subgroup of order $p^2$ contained in $G'$. Now the proof follows from the preceding lemma. \hfill $\Box$

\end{proof}

In particular, if, for an odd prime $p$,  we take a $p$-group $G$ from the class $\hat{G}_{3}$ such that the exponent of $G$ is $p$, then $G$ is isomorphic to the group $G_{3}$ defined in  \eqref{6eqn1} for $r = 3$. Then by Lemma \ref{ext-automorphism}
a bijection between any two minimal generating sets for  $G_{3}$,  extends to an automorphism of $G$. Therefore, on the lines of the proofs of Lemma \ref{lem7} and Corollary \ref{corollary2} (replacing autoclinism by automorphism and isoclinic by isomorphic),  we can prove the following result, which has also been proved by Brahana \cite[Section 2]{Brahana40}. But the proof in the  present text is in modern terminology.

\begin{lemma}\label{4-gen-lemma}
Let $K$ be a  $4$-generator special $p$-group of order $p^{8}$, exponent $p$ and conjugate type $\{1,p^{3}\}$, where $p$ is an odd prime. Then $K$ is isomorphic to $G_{3}/N$, where $G_{3}$ is the group  defined  in  \eqref{6eqn1}  generated by $a,\;b,\;c,\;d$ and $N < Z(G) = G'$ with $|N| = p^2$ is of the form
$$N = \gen{ [a,\;b][c,\;d],\; [a,\;c][b,\;d]^r},\; \mbox{where}\; r\; \mbox{is any fixed non-square integer modulo}\; p.$$
\end{lemma}

Now we consider the family $\hat{\mathcal{G}}_3$ of $2$-groups defined in the introduction. We start with the following result which tells that certain type of quotient groups of any two groups in  $\hat{\mathcal{G}}_3$ are isoclinic.

\begin{lemma}\label{lem8}
Let  $G  = \gen{a,\;b,\;c,\;d}$ and $G^*  = \gen{s,\;u,\;v,\;w}$ be two groups  from the family $\hat{\mathcal{G}}_3$. Then the following hold true.
\begin{enumerate}
\item[(i)] $G$ and $G^*$  are isoclinic.
\item[(ii)] If $M = \gen{ [a,\; b][c,\; d]} \le G'$ and $M^* = \gen{ [s,\; u][v,\; w]} \le (G^*)'$, then $G/M$ and $G^*/M^*$ are isoclinic.
\item[(iii)] If $$N = \gen{[a,\; b][c,\; d], \; [a,\;c][b,\;d][c,\;d]} \le G'$$ and $$N^* = \langle[s,\; u][v,\; w], \; [s,\;v][u,\;w][v,\;w]\rangle \le (G^*)',$$ then $G/N$ are $G^*/N^*$ are isoclinic.
\end{enumerate}
\end{lemma}

\begin{proof}
We sketch  proof only for $(i)$.  Note that both  $G/\Z(G)$ and $G^*/\Z(G^*)$  are elementary abelian $2$-groups of order $2^4$, generated by $\{a\Z(G), b\Z(G), c\Z(G),$ $ d\Z(G) \}$ and $\{s\Z(G^*), u\Z(G^*), v\Z(G^*), w\Z(G^*)\}$ respectively. Similarly, both  $G'$ and $(G^*)'$ are   
elementary abelian $2$-groups generated by the sets consisting of all $6$ basic commutators 
$$\{[a,\;b], [a,\;c], [a,\;d], [b,\;c], [b,\;d], [c,\;d]\}$$ 
and 
$$\{[s,\;u], [s,\;v], [s,\;w], [u,\;v], [u,\;w], [v,\;w]\}$$ 
respectively. Now the  map $a \mapsto s$, $b \mapsto u$, $c \mapsto v$ and $d \mapsto w$ extends to an isomorphism from $G/\Z(G)$ onto $G^*/\Z(G^*)$, which induces an isomorphism from $G'$ onto $(G^*)'$, making $G$ and $G^*$ isoclinic. 
\hfill $\Box$

\end{proof}

We remark that the second and third  assertions of the preceding lemma  hold true in the bigger family $\hat{G}_3$. And by the same argument as given in the proof, one can easily  prove that any two groups in  $\hat{G}_n$ are isoclinic.  We have stated this result  for the family $\hat{\mathcal{G}}_3$ because we here need it only for $2$-groups.

\smallskip
The following lemma is immediate from Corollary \ref{corollary1}, using Lemma \ref{lem8},  when restricted to the family $\hat{\mathcal{G}}_3$.

\begin{lemma}\label{lem9}
Let $K$ be a  $4$-generator  special $2$-group of order  $2^{9}$ and conjugate type $\{1,8\}$. Then $K$ is isoclinic to $G/M$, where $G = \gen{a,\;b,\;c,\;d} \in \hat{\mathcal{G}}_{3}$ is any fixed group  and $M < Z(G) = G'$ with $|M| = 2$ is of the form
$M = \langle [a,\; b][c,\; d]\rangle$.

\end{lemma}

The following lemma is analogous to Lemma \ref{lem7} for $p=2$, and therefore the proof is mostly a duplication of the the proof of Lemma \ref{lem7} with necessary modifications.

\begin{lemma}\label{lem10}
Let $G = \gen{a, b, c, d} \in \hat{\mathcal{G}}_3$. Then $G/N$ with $|N| = 4$ is of conjugate type $\{1, 8\}$ if and only $N$ can be reduced to the form  $N =  \gen{[a,\; b][c,\; d], \; [a,\;c][b,\;d][c,\;d]}$. 
\end{lemma}

\begin{proof}
Notice  that $|G| = 2^{10}$ and $|G'| = 2^6$. Both $G/\Z(G)$ and $G' = \Z(G)$ are elementary abelian.  Set $\overline{G}$ = $G/N$; then $|\overline{G}|$ = $2^8$. Since $G$ is of conjugate type $\{1, 8\}$ and $|N| = 4$, it follows that 
$$ \Z(G)/N = \Z(\overline{G}) = \overline{G}'= \Phi(\overline{G})$$
 is an elementary abelian $2$-group of order $2^4$. Thus $[\overline{G} : \Z(\overline{G})] = 2^4$. Hence $\overline{G}$ is of conjugate type $\{1, 8\}$ if and  only if each non-central element of $\overline{G}$ commutes only with its own powers up to the central elements.

Let $\bar{x},\; \bar{y} \in  \overline{G} - \Z(\overline{G})$ be such that no one is a power of the other (reading modulo $\Z(\overline{G})$). Then it is not difficult to see that  $[x,y] \neq 1$ in $G$. Hence, 
if $[\overline{x} , \overline{y}] =1$ in $\overline{G}$, then  $[x,y] \in N^{\#}$.  Any given central subgroup $N_1$ of order $4$, without loss of generality,   can be written   as one of the following two types:
\begin{eqnarray*}
(i) \;\;\; N_1 &=& \langle [a,\;b][a,\;d]^{i_1}[b,\;c]^{i_2}[b,\;d]^{i_3}[c,\;d]^{i_4},\; [c,\; d] \rangle.\\
(ii) \;\;\; N_1 &=& \langle [a,\;b][a,\;d]^{i_1}[b,\;c]^{i_2}[b,\;d]^{i_3}[c,\;d]^{i_4},\; [a,\;c][a,\;d]^{j_1}[b,\;c]^{j_2}[b,\;d]^{j_3}[c,\;d]^{j_4} \rangle.
\end{eqnarray*}

If $N_1$ is  of type (i), then $\bar{c}$  commutes with $\bar{d}$, although $\overline{c}$ $\notin \langle Z(\overline{G}), \overline{d} \rangle$. Hence $\overline{G}$ can not  be of conjugate type $\{1,p^3\}$. 
Therefore we only need to consider $N_1$ as in type (ii). Now, as done in the proof of  Lemma \ref{lem7}, we can reduce $N_1$ to the form
\begin{equation*}
N_2 = \langle [a,\;b][c,\;d],\; [a,\;c][b,\;d]^{i_1}[c,\;d]^{i_2} \rangle.
\end{equation*}
Here $i_1$ can not be $0$, else $c$ will commute with $ad^{-i_2}$; so $i_1 = 1$. Again $i_2$ can not be $0$, else $ad^{-1}$ will commute with $bc$; so $i_2 = 1$, and hence $N_2 = \langle [a,\;b][c,\;d],\; [a,\;c][b,\;d][c,\;d]\rangle$. 

\smallskip

Now consider $x = a^{j_1}b^{j_2}c^{j_3}d^{j_4}$ and $y = a^{k_1}b^{k_2}c^{k_3}d^{k_4}$ be such that none is power of the other (reading modulo $\Z(G)$). If $[x,\;y] \in N_2^{\#}$, then, on the lines of the proof of Lemma \ref{lem7}, it follows that
$[a,\;b][a,\;c]^k[b,\;d]^j[c,\;d]^{jk}$ $\in N_2^{\#}$, and therefore can be written as a product of powers of generators of $N_2^{\#}$. 

Now comparing powers of the basic commutators, we  get

\begin{equation*}
j \equiv k \;(\mbox{mod}\; 2)\; \mbox{and}\; jk\equiv k + 1\; (\mbox{mod}\; 2)
\end{equation*}
Solving these we  have
\begin{equation*}
k^2 - k - 1 \equiv 0 \;( \mbox{mod}\; 2).
\end{equation*}
This is not possible. Hence no non-central element commutes with other elements except its power (modulo center) in $G / N_2$ if and only if $N_2 = \gen{[a,\;b][c,\;d],\; [a,\;c][b,\;d][c,\;d]}$. \hfill $\Box$

\end{proof}

Now using Lemma \ref{lem8}, the preceding lemma gives
\begin{cor}\label{corollary3}
Let $K$ be a  $4$-generator  special $2$-group of order  $2^{8}$ and conjugate type $\{1,8\}$. Then $K$ is isoclinic to $G/N$, where $G = \gen{a,\;b,\;c,\;d}  \in  \hat{\mathcal{G}}_3$  be any fixed group and $N < Z(G) = G'$ with $|N| = 4$ is of the form
$$N =  \gen{[a,\; b][c,\; d], \; [a,\;c][b,\;d][c,\;d]}.$$
\end{cor}

We conclude this section with the following result which is valid only for odd primes.
\begin{lemma}\label{expresult}
Every isoclinism family of finite p-groups of nilpotency class $2$ and  conjugate type $\{1, p^n\}$ contains a group of exponent $p$, where $p$ is an odd prime.
\end{lemma}

\begin{proof}
Notice that the isoclinism family of a  finite p-group $G$ of nilpotency class $2$ and  conjugate type $\{1, p^n\}$ contains a special $p$-group $H$ (say).  Then $H$ has the following presentation.
\begin{equation*}
\begin{split}
H=\Big{\langle} x_1,x_2,\ldots, x_d \colon & [x_i,x_j,x_k]=1,  [x_i,x_j]^p=1,\\
 & x_i^p=\prod_{j<k} [x_j,x_k]^{c_{ijk}},  \prod_{j<k} [x_j,x_k]^{d_{ljk}}=1\Big{\rangle},
\end{split}
\end{equation*}
where $c_{ijk}, d_{ljk}\in\mathbb{Z}$.  Let $F/R$ be a free presentation of $H$, and $R_1$ denote the subgroup of $R$ which is the normal closure of $\{[x_i,x_j,x_k], [x_i,x_j]^p, \prod_{j<k} [x_j,x_k]^{d_{ljk}}\}$ in $F$. Let $\overline{F} := F/R_1$. Then the group $\overline{F}/\overline{F}^p$ lies in the isoclinism class of $G$ and is of exponent $p$. \hfill $\Box$

\end{proof}

\section{Proof of Theorems~\ref{thm1} and \ref{thm2}}

We are now ready to prove our main results.\\

\noindent {\it Proof of Theorem \ref{thm1}.}
Let $G$ be a finite $p$-group of conjugate type $\{1,p^3\}$, $p > 2$. Then by  \cite[Main Theorem]{Ishikawa2002}, $G$ can be of nilpotency class $2$ or $3$. Without loss of any generality, we can always assume that $\Z(G) \le G'$. First assume that  $G$ is of class $3$. We are going to show that this case can not occur, and therefore $G$ must have nilpotency class $2$.
 
By Proposition \ref{prop7}, $|Z(G)|\geq p^3$ and  $Z(G) < G'$; so $|G'| \geq p^4$. Then it follows from  Lemma \ref{lem3} that $[G : Z(G)] = p^4$. Since  $Z(G) < G'$; we have $[G : G'] \leq p^3 $. But, if $[G : G'] \leq p^2 $, then $G$ can be minimally  generated by at most $2$ elements, which  contradicts
Proposition \ref{prop7}. Thus $\mid G : G' \mid = p^3 $ and minimal generating set for $G$ has exactly $3$ elements. Assume that  $G$ = $\langle a,b,c \rangle $.

 Now we have $[G : Z(G)] = p^4$ and $[G : G'] = p^3 $. So, at least one of the three commutators $[a,\;b], [a,\;c]$ and $[b,\;c]$ lies outside center. 
By the symmetry, we can assume that  $[a,\;b] \in G' - Z(G)$. Set  $[a,\;b] = \alpha$. Then clearly $G'$ = $\langle \alpha, Z(G)\rangle$. So there exist integers $i_1$ 
and $i_2$ such that
\begin{equation*}
[a,\;c] = [a,\;b]^{i_1}\beta_1,\; \mbox{where}\; \beta_1\in Z(G)
\end{equation*}
and
\begin{equation*}
[b,\;c] = [a,\;b]^{i_2}\beta_2,\; \mbox{where}\; \beta_2\in Z(G).
\end{equation*}
Replacing $c$ by $a^{i_2}b^{-i_1}c$, we  get $[a,\;c], [b,\;c] \in Z(G)$. Then $[\alpha,\;c] = 1$ (by Lemma \ref{lem1}). An arbitrary element of $G$ can be written 
as $g = a^{j_1}b^{j_2}c^{j_3}\alpha_1^{j_4}z$; where $z \in G' = Z(G)$ and $0\leq j_k \leq p-1$ for $k = 1,2,3,4$. Then
\begin{equation*}
\alpha^G =  \{ (a^{j_1}b^{j_2})^{-1}\alpha \;a^{j_1}b^{j_2} \mid  0\leq j_k \leq p-1,\;  k = 1, 2\}.
\end{equation*}  
Thus $|\alpha^G|$ $\leq p^2$, which contradicts the fact that $G$ is of conjugate type $\{1,p^3\}$. Hence the nilpotency class of $G$ must be $2$.

Now onward we assume that the nilpotency class of $G$ is $2$. Recall that  $G/Z(G)$ and $G'$ are elementary abelian p-groups (by Lemma \ref{lem2}). By our assumption that $Z(G) \le G'$,  we have $Z(G) = G' = \Phi(G)$. By Proposition \ref{prop4} 
and Proposition \ref{prop5}, we have $p^3 \leq  |Z(G)| = |G'| \leq p^6$. Thus, by Lemma \ref{lem3}, there can  be two possibilities, namely
\begin{enumerate}[(i)]
 \item[$(i)$] $|G'|$ = $p^3$ and $[G : Z(G)] \geq p^4 $ or 
 \item[$(ii)$] $|G'| \geq p^4$ and $[G : Z(G] = p^4$.
\end{enumerate}
In case (i), $G$ is a Camina group with $|G'|$ = $p^3$. 

\smallskip
So it remains to consider  case (ii) only.
In this case,  we have  $ p^4 \leq |G'| = |Z(G)| \leq p^6$ and $[G : Z(G)] = p^4 = [G : G'] = [G: \Phi(G)]$. Thus $G/ {\Phi(G)}$ is an elementary abelian 
$p$-group of order $p^4$. Hence $G$ is minimally generated by $4$ elements. By Lemma \ref{expresult} we can assume $G$ to be of exponent $p$ upto isoclinism. 
Thus $G$ is isoclinic to $G_3$ or to  a central quotient $G_3/H$, where $H$ is a non-trivial central subgroup of $G_3$  with $|H| \leq p^2$. Hence the order 
of $H$ is  either $p$ or $p^2$.  Proof of the theorem is now complete by Corollary  \ref{lem6} and Lemma \ref{4-gen-lemma}. \hfill $\Box$

\smallskip
Before proceeding to the proof of Theorem \ref{thm2}, we state the following result which is a consequence of the main result of Wilkens \cite{Wilkens07} stated on pages $203 - 204$.

\begin{theorem}\label{Wilkens-thm}
Let $G$ be a finite $2$-group of nilpotency class $2$ and conjugate type $\{1, 8\}$. Then one of the following holds:
\begin{enumerate}
 \item[(i)] $|G'| = 2^3$. 
 \item[(ii)] $[G : \Z(G)] = 2^4 $.
 \item[(iii)] $|G'| = 2^4$ and there exists $R$ with $R \leq \Omega_1(\Z(G))$ and $|R| = 2$ such that $|G/R : \Z(G/R)| = 2^3$.
 \item[(iv)] $|G'| = 2^4$ and $G$ is central product $HC_G(H)$, where $C_G(H)$ is abelian and $H$ is  the group given as  follows:
 
  There are $i$, $j$, $k$, $l$ and $m$ $\in \mathbb{N}$ such that $H \cong$ $\tilde{H}/\gen{x^{2i}, v^{2j}, v_1^{2k}, v_2^{2l}, v_3^{2m}}$, where $\tilde{H} = \gen{x, v, v_1, v_2, v_3 }$ is of class $2$ with $\Phi(\tilde{H})
  \leq \Z(\tilde{H})$ and is otherwise defined by $[v_2, x] = [v_1, v] = [v_3, x][v_3, v] = 1$, $[v_i, v_j] \in \gen{[v_3, x]}$.
\end{enumerate} 
\end{theorem}

We are now  ready for the final proof. 

\noindent {\it Proof of Theorem \ref{thm2}.}
Let $\mathcal{G}$ be a finite $2$-group of nilpotency class $2$ and conjugate type $\{1, 8\}$. Then $\mathcal{G}$ is isomorphic to one of the groups  $G$ in (i), (ii), (iii) and (iv) of the preceding theorem. We are going to show that third and fourth possibilities can not occur. Suppose that (iii) occurs. Then $|G'| = 2^4$ and there exists $R$ with $R \leq \Omega_1(Z(G))$ and $|R| = 2$ such that $|G/R : \Z(G/R)| = 2^3$. Since $G$ is of conjugate type $\{1, 8 \}$ and $|R| = 2$, we have  $\Z(G/R) = \Z(G)/R$. Then from the fact that $|G/R : Z(G/R)| = 2^3$, we  get $|G : \Z(G)| = 2^3$, which  contradicts our hypothesis that $G$ is of conjugate type $\{1, 8 \}$.

Next consider the case (iv)(5).  So $G \cong HC_G(H)$, where $\tilde{H}$ is a $2$-group of class $2$. It is easy to see that the conjugacy class of the image of $v_3$ in $H$ is of lenght at most $2$. Hence $H$ is not of conjugate type $\{1, 8\}$ and so is for $G$. So we are left with only two cases (i) and (ii).

In case  (i),   $|G'| = 2^3$, which forces  $\mathcal{G}$ to be   isoclinic to  a  Camina $2$-group with commutator subgroup of order $8$. 

Finally we consider the case  (ii). For any group $G$, in this case, we have  $[G : \Z(G)] = 2^4$. Since $G$ is of conjugate type $\{1, 8\}$, we have  $[G : C_G(x)] = 2^3$, and consequently $[C_G(x) : \Z(G)] = 2$ for all $x \in G-\Z(G)$. Hence for all $x \in G-\Z(G)$, $x^2 \in \Z(G)$, 
i.e., $G/\Z(G)$ is an elementary abelian $2$-group. Thus $G' \leq \Z(G)$, and therefore  $G$ is of class $2$. 

By Proposition \ref{prop5} we can assume  $Z(G) = G'$. By Proposition \ref{prop6}, $|G'| = |Z(G)| \leq 2^6$. Since,   $G$ being conjugate type $\{1, 8\}$,  $|G'| \geq 2^3$, it follows that $2^7 \leq |G| = 2^{10}$. Since $G$ is 
of class $2$, obviously $G' = Z(G)$ is  elementary abelian. Therefore the exponent of $G$ is $4$. If $|G| = 2^7$, then $G$ is a Camina group, which is not possible by \cite[Theorem 3.2]{Macdonald1981}. Hence $2^8 \le |G| \le 2^{10}$, and therefore $G$ must be isomorphic to some group $T$ in the family $\hat{\mathcal{G}}$  or its central quotient $T/K$ with $|K| \leq 4$. Now the proof is complete by  Lemmas \ref{lem8}, \ref{lem9} and Corollary \ref{corollary3}. \hfill $\Box$

\hspace{.4in}

\noindent{\bf Acknowledgements.} The authors thank Prof. Mike Newman for various useful suggestions and comments. Reference \cite{Brahana40}  was brought to the attention of the authors by him.

\end{document}